\documentclass[11pt,a4paper]{article}
\usepackage[T1]{fontenc}
\usepackage[utf8]{inputenc}
\usepackage{lmodern,amsmath,amssymb,amsthm,mathtools,booktabs,array,microtype}
\usepackage[margin=27mm]{geometry}
\usepackage[numbers,sort&compress]{natbib}
\usepackage[colorlinks=true,linkcolor=blue,citecolor=blue,urlcolor=blue]{hyperref}
\hypersetup{pdftitle={Biderivations of Null-Filiform and Filiform Zinbiel Algebras},pdfauthor={Arislanbay Abdireymov and Bakhtiyor Yusupov}}
\newtheorem{theorem}{Theorem}[section]
\newtheorem{proposition}[theorem]{Proposition}
\newtheorem{lemma}[theorem]{Lemma}
\newtheorem{corollary}[theorem]{Corollary}
\theoremstyle{definition}
\newtheorem{definition}[theorem]{Definition}
\newtheorem{example}[theorem]{Example}
\theoremstyle{remark}\newtheorem{remark}[theorem]{Remark}
\DeclareMathOperator{\Ann}{Ann}\DeclareMathOperator{\Der}{Der}
\DeclareMathOperator{\Bider}{Bider}\DeclareMathOperator{\End}{End}
\DeclareMathOperator{\Hom}{Hom}\DeclareMathOperator{\rank}{rank}
\DeclareMathOperator{\tr}{tr}\DeclareMathOperator{\rad}{rad}
\newcommand{\bsym}{\Bider_{\mathrm{sym}}}\newcommand{\bskew}{\Bider_{\mathrm{skew}}}
\newcommand{\bann}{\Bider_{\Ann}}\newcommand{\NF}{\mathrm{NF}}
\newcommand{\K}{\mathbb K}\newcommand{\C}{\mathbb C}
\newcommand{\spn}[1]{\langle #1\rangle}
\numberwithin{equation}{section}
\title{Biderivations of Null-Filiform and Filiform Zinbiel Algebras}
\author{Arislanbay Abdireymov\thanks{V.~I.~Romanovskiy Institute of Mathematics, Academy of Sciences of Uzbekistan, Tashkent, Uzbekistan. Email: \texttt{aaabdireymov1001@gmail.com}.}
\and Bakhtiyor Yusupov\thanks{Corresponding author. V.~I.~Romanovskiy Institute of Mathematics, Academy of Sciences of Uzbekistan, Tashkent, Uzbekistan; Department of Applied Mathematics and Mathematical Physics, Urgench State University, Urgench, Uzbekistan. Email: \texttt{baxtiyor\_yusupov\_93@mail.ru}.}}
\date{}
\begin{document}\maketitle
\begin{abstract}
A biderivation of an algebra is a bilinear map which is a derivation in each argument. We describe all biderivations of finite-dimensional complex null-filiform and filiform Zinbiel algebras. For the $n$-dimensional null-filiform algebra, every biderivation is determined by an arbitrary value at its generating pair; the space of biderivations has dimension $n$ and consists entirely of symmetric maps. For the three standard filiform families, we give explicit bases and determine their symmetric and skew-symmetric parts. We also treat all filiform algebras of dimensions at most four, including the exceptional families not covered by the classification in dimensions at least five. A three-dimensional algebra admits a three-dimensional space of skew-symmetric biderivations, two of which have values outside its annihilator. Up to isomorphism, it is the unique complex filiform Zinbiel algebra with this property. The proofs use generators, annihilator-valued maps, a truncated polynomial model, and conformal endomorphisms of bilinear forms.
\end{abstract}

\section{Introduction}
Zinbiel algebras arose in Loday's study of the cup product on Leibniz cohomology \cite{Loday}. Their defining identity is dual, in the operadic sense, to the Leibniz identity. The cohomological background is developed by Loday and Pirashvili \cite{LP}; Koszul duality and its general framework are treated by Ginzburg and Kapranov \cite{GK} and Loday and Vallette \cite{LV}. Although Zinbiel multiplication is generally neither associative nor commutative, its symmetrization is associative and commutative. This combination makes questions about derivations and related maps accessible through several complementary methods.

The nilpotent structure of finite-dimensional Zinbiel algebras is a central part of their theory. Dzhumadil'daev and Tulenbaev studied nilpotency in characteristic zero \cite{DT}; Towers proved nilpotency over an arbitrary field \cite{Towers}. Identities in dual Leibniz algebras were investigated by Naurazbekova and Umirbaev \cite{NU}. Classifications under restrictions on the nilindex, gradation, or characteristic sequence provide explicit multiplication tables suitable for studying linear and multilinear operators. In particular, null-filiform and filiform algebras were described by Adashev, Khudoyberdiyev and Omirov \cite{AKO}. Related classification results include naturally graded Zinbiel algebras with nilindex $n-3$ \cite{ACGK} and those with characteristic sequence $(n-p,p)$ \cite{ALO}.

A bilinear map which is a derivation separately in each variable is called a biderivation. Such maps have a substantial history in associative ring theory. Bre\v sar studied generalized biderivations and their relation to commuting maps \cite{Bresar95,Bresar04}; subsequent results include biderivations of triangular algebras \cite{Benkovic} and generalized matrix algebras \cite{DuWang}. In Lie algebra theory, biderivations have been studied on parabolic subalgebras of simple Lie algebras \cite{WYC}, finite-dimensional complex simple Lie algebras \cite{Tang}, and Block Lie algebras \cite{LGZ}. Bre\v sar and Zhao developed connections with commuting linear maps \cite{BZ}. More recently, Chen, Yao and Zhao studied symmetric biderivation radicals and used them to determine biderivations for several Lie algebra classes \cite{CYZ}. The corresponding superalgebra setting was investigated by Yuan and Tang \cite{YT}.

The bilinear notion used here is distinct from the notion of a pair of linear maps that is also called a biderivation in parts of the Leibniz algebra literature. We do not impose symmetry in the definition. Instead, we determine the full space and then identify its symmetric and skew-symmetric subspaces. Results for Lie or associative algebras do not directly determine these spaces for Zinbiel algebras, whose multiplication is neither skew-symmetric nor generally associative.

The purpose of this paper is to give explicit classifications for null-filiform and filiform Zinbiel algebras. The null-filiform algebra is generated by one element, and for dimension at least five the complex filiform algebras belong to three families described in \cite{AKO}. These facts suggest reducing a bilinear map to its values on pairs of generators. This determines uniqueness, but it does not guarantee that arbitrary generator values extend. Establishing the allowable values is the central step of the classification.

For the null-filiform algebra we use a truncated polynomial realization. Derivations have the form $f\mapsto pf'$ and biderivations have the form $(f,g)\mapsto rf'g'$. This gives both a classification and a direct explanation of their symmetry. For the standard filiform families, annihilator-valued maps account for all skew-symmetric biderivations, while symmetric maps supply the remaining directions. The main dimension formulas are
\begin{center}\begin{tabular}{@{}lcccc@{}}\toprule
$A$ & Range & $\dim\Bider(A)$ & $\dim\bsym(A)$ & $\dim\bskew(A)$\\\midrule
$\NF_n$ & $n\ge1$ & $n$ & $n$ & $0$\\
$F_n^1$ & $n\ge3$ & $n+6$ & $n+4$ & $2$\\
$F_n^2$ & $n\ge4$ & $n+3$ & $n+2$ & $1$\\
$F_n^3$ & $n\ge4$ & $n+1$ & $n$ & $1$\\\bottomrule
\end{tabular}\end{center}
Small dimensions require separate treatment: the three standard families do not exhaust dimension three or four. We classify the filiform algebras in those dimensions and compute their biderivations. The three-dimensional case reduces to bilinear forms on a two-dimensional space. One alternating form produces an exceptional algebra for which skew-symmetric biderivations need not take values in the annihilator. This identifies the precise limitation of a structural property shared by the standard families.

The paper is organized as follows. Section~2 establishes general properties. Sections~3 and~4 treat the null-filiform algebra and the standard filiform families. Sections~5 and~6 address dimensions three and four. Section~7 collects consequences and describes exact computational checks, and Section~8 presents further questions. The arbitrary-dimensional results are proved mathematically; the accompanying rational-arithmetic script provides supplementary finite-dimensional checks.

\section{Preliminaries and general properties}
Unless otherwise indicated, $\K$ is a field of characteristic zero. Classifications up to isomorphism are stated over $\C$.
\begin{definition}
A Zinbiel algebra is a vector space $A$ with a bilinear multiplication $\circ$ satisfying
\begin{equation}\label{zin}(x\circ y)\circ z=x\circ(y\circ z)+x\circ(z\circ y),\qquad x,y,z\in A.\end{equation}
A linear map $D:A\to A$ is a derivation if $D(x\circ y)=D(x)\circ y+x\circ D(y)$. A bilinear map $B:A\times A\to A$ is a biderivation if
\begin{align}
B(x\circ y,z)&=B(x,z)\circ y+x\circ B(y,z),\label{bi1}\\
B(x,y\circ z)&=B(x,y)\circ z+y\circ B(x,z)\label{bi2}
\end{align}
for all $x,y,z\in A$.
\end{definition}
We denote the spaces of all, symmetric, and skew-symmetric biderivations by $\Bider(A)$, $\bsym(A)$, and $\bskew(A)$. Put
\[A^2=\operatorname{span}\{x\circ y:x,y\in A\},\qquad \Ann(A)=\{a:a\circ A=A\circ a=0\}.\]
For a basis $e_1,\ldots,e_n$, let $E_{p,q}^v$ take $(e_p,e_q)$ to $v$ and all other basis pairs to zero. All omitted products and omitted values of specified maps are zero.
\subsection{Symmetry, annihilators, and generators}
The next results hold for any algebra with a bilinear multiplication.
\begin{proposition}\label{sym}
The spaces $\Bider(A)$, $\bsym(A)$, and $\bskew(A)$ are vector spaces. If $B^\tau(x,y)=B(y,x)$, then $B\in\Bider(A)$ implies $B^\tau\in\Bider(A)$. Consequently,
\[\Bider(A)=\bsym(A)\oplus\bskew(A).\]
A symmetric or skew-symmetric bilinear map satisfying \eqref{bi1} also satisfies \eqref{bi2}.
\end{proposition}
\begin{proof}
The defining identities are linear in $B$. Interchanging the arguments converts \eqref{bi1} for $B^\tau$ into \eqref{bi2} for $B$, and conversely. Thus $B=\frac12(B+B^\tau)+\frac12(B-B^\tau)$ is the asserted decomposition. The intersection is zero because $2B=0$ for a map which is both symmetric and skew-symmetric. If $B^\tau=\pm B$, the two identities are equivalent.
\end{proof}
\begin{proposition}\label{preserve}
Every biderivation satisfies
\[B(A^2,A)+B(A,A^2)\subseteq A^2,\quad B(\Ann(A),A)+B(A,\Ann(A))\subseteq\Ann(A).\]
\end{proposition}
\begin{proof}
The first inclusion follows from \eqref{bi1}--\eqref{bi2}. If $a\in\Ann(A)$, then
\[0=B(a\circ x,z)=B(a,z)\circ x,\qquad 0=B(x\circ a,z)=x\circ B(a,z).\]
Hence $B(a,z)\in\Ann(A)$. The second argument is treated using \eqref{bi2}.
\end{proof}
\begin{theorem}\label{ann}
Let $\bann(A)=\{B\in\Bider(A):B(A,A)\subseteq\Ann(A)\}$. Then:
\begin{enumerate}
\item A bilinear map $B:A\times A\to\Ann(A)$ is a biderivation if and only if $B(A^2,A)=B(A,A^2)=0$.
\item There is a natural linear isomorphism
\[\bann(A)\cong\Hom\big((A/A^2)\otimes(A/A^2),\Ann(A)\big).\]
\item If $d=\dim(A/A^2)$ and $a=\dim\Ann(A)$ are finite, then
\begin{align*}
\dim\bann(A)&=d^2a,\\
\dim(\bann(A)\cap\bsym(A))&=\frac{d(d+1)}2a,\\
\dim(\bann(A)\cap\bskew(A))&=\frac{d(d-1)}2a.
\end{align*}
\end{enumerate}
\end{theorem}
\begin{proof}
If the image is in the annihilator, the right-hand sides of both identities vanish. The identities therefore say exactly that $B$ vanishes whenever either argument belongs to $A^2$. Such maps factor uniquely through $(A/A^2)\times(A/A^2)$; conversely, every bilinear map on this quotient with values in $\Ann(A)$ lifts to a biderivation. Choosing bases gives arbitrary, symmetric, and skew-symmetric coefficient arrays with $d^2a$, $d(d+1)a/2$, and $d(d-1)a/2$ independent entries.
\end{proof}
\begin{proposition}\label{generators}
If $S$ generates $A$ as an algebra, a biderivation is uniquely determined by its restriction to $S\times S$. A derivation is uniquely determined by its values on $S$.
\end{proposition}
\begin{proof}
It suffices to prove that vanishing on the specified set implies vanishing everywhere. If $B(S,S)=0$, fix $s\in S$. The kernel of the derivation $B(s,\cdot)$ is a subalgebra containing $S$, hence equals $A$. Thus $B(S,A)=0$. For each $y\in A$, the kernel of $B(\cdot,y)$ now contains $S$ and again equals $A$. The derivation assertion follows from the same argument.
\end{proof}
\subsection{Powers and filiform algebras}
For a Zinbiel algebra define $A^1=A$ and
\[A^k=\sum_{i=1}^{k-1}A^i\circ A^{k-i}\quad(k\ge2).
\]
These powers satisfy $A^i\circ A^j\subseteq A^{i+j}$ and $A^{k+1}\subseteq A^k$.
\begin{lemma}\label{powers}
For a Zinbiel algebra, $A^k=A\circ A^{k-1}$ for every $k\ge2$.
\end{lemma}
\begin{proof}
The inclusion from right to left is in the definition. Induct on $k$ for the converse. If $i\ge2$, induction gives $A^i=A\circ A^{i-1}$. By \eqref{zin}, a spanning element $(x\circ u)\circ v$ of $A^i\circ A^{k-i}$ equals $x\circ(u\circ v+v\circ u)$ and belongs to $A\circ A^{k-1}$. The term $i=1$ already has this form. The descending property follows inductively from the definition, starting with $A^2\subseteq A$.
\end{proof}
\begin{definition}
An $n$-dimensional Zinbiel algebra is null-filiform if $\dim A^k=n+1-k$ for $1\le k\le n+1$. It is filiform if $\dim A^k=n-k$ for $2\le k\le n$.
\end{definition}
For filiform algebras of dimension $n\ge2$, $\dim(A/A^2)=2$. For $n=1$ the condition is empty; we include the unique one-dimensional Zinbiel algebra as a degenerate case. This convention does not affect dimensions at least two.

\section{Biderivations of the null-filiform algebra}
By \cite{AKO}, the null-filiform algebra $\NF_n$ has a basis $e_1,\ldots,e_n$ with
\begin{equation}\label{nf}e_i\circ e_j=\binom{i+j-1}{j}e_{i+j}\quad(i+j\le n).\end{equation}
It is generated by $e_1$, since $e_1\circ e_i=e_{i+1}$.
\subsection{A polynomial model and its derivations}
Consider
\[P_n=t\K[t]/t^{n+1}\K[t],\qquad f\circ g=\int_0^t f'(s)g(s)\,ds\pmod{t^{n+1}}.\]
Integration is formal polynomial integration in characteristic zero. Changing either input by an element of $t^{n+1}\K[t]$ changes the output by such an element, so the product is well defined. The map $e_i\mapsto t^i/i!$ identifies \eqref{nf} with this product. Moreover, $g\circ h+h\circ g=gh$, so
\[x\circ(g\circ h+h\circ g)=\int_0^t x'gh=\int_0^t(x\circ g)'h=(x\circ g)\circ h.\]
Thus $P_n$ is Zinbiel.
\begin{lemma}\label{polyder}
Every derivation of $P_n$ has a unique expression $D_p(f)=pf'$, with $p\in P_n$. In particular, $\dim\Der(\NF_n)=n$.
\end{lemma}
\begin{proof}
Differentiation alone does not descend to the quotient, but if $p(0)=0$, then $p(t)(t^{n+1}q)'$ is divisible by $t^{n+1}$. Thus $D_p$ is well defined and depends only on the class of $p$. For representatives,
\begin{align*}
(pf')\circ g+f\circ(pg')
&=\int_0^t\big((pf')'g+f'pg'\big)\,ds\\
&=\int_0^t(pf'g)'\,ds=pf'g=D_p(f\circ g).
\end{align*}
The boundary term vanishes because $p(0)=0$. Conversely, $t$ generates the algebra, so $D$ is determined by $D(t)$. Since $D_{D(t)}(t)=D(t)$, Proposition~\ref{generators} gives $D=D_{D(t)}$. Evaluation at $t$ proves uniqueness.
\end{proof}
\subsection{The classification theorem}
For $1\le k\le n$ define
\begin{equation}\label{Bk}B_k(e_i,e_j)=\frac{(i+j+k-2)!}{k!(i-1)!(j-1)!}\,e_{i+j+k-2},\end{equation}
where $e_\ell=0$ for $\ell>n$.
\begin{theorem}\label{null}
Every biderivation of $\NF_n$ has a unique expression
\[B=\sum_{k=1}^n r_kB_k,\qquad B(e_1,e_1)=\sum_{k=1}^n r_ke_k,\]
with arbitrary $r_k\in\K$. Equivalently,
\[B(e_i,e_j)=\sum_{k=1}^{\min\{n,n-i-j+2\}}r_k\frac{(i+j+k-2)!}{k!(i-1)!(j-1)!}\,e_{i+j+k-2},\]
where a sum with upper bound less than $1$ is zero. Consequently,
\[\dim\Bider(\NF_n)=\dim\bsym(\NF_n)=n,\qquad\bskew(\NF_n)=0.\]
\end{theorem}
\begin{proof}
In $P_n$, for $r\in P_n$ define $B_r(f,g)=rf'g'$. Changing either input by a multiple of $t^{n+1}$ changes its derivative by a multiple of $t^n$; multiplication by $r$, whose constant term is zero, restores divisibility by $t^{n+1}$. Changes in the representative of $r$ also have no effect. For fixed $g$, this is $D_{rg'}$, a derivation by Lemma~\ref{polyder}. Symmetry then makes it a biderivation.

Conversely, put $r=B(t,t)$. Lemma~\ref{polyder} gives $B(f,g)=B(t,g)f'$. Since $B(t,\cdot)$ is a derivation taking $t$ to $r$, $B(t,g)=rg'$. Hence $B(f,g)=rf'g'$. Taking $r=t^k/k!$ yields \eqref{Bk}. Evaluation at $(t,t)$ proves independence and uniqueness. All resulting maps are symmetric; Proposition~\ref{sym} gives the skew-symmetric assertion.
\end{proof}
\begin{remark}
The theorem includes $n=1,2$. Also $\Ann(\NF_n)=\spn{e_n}$ and $\NF_n/\NF_n^2=\spn{\bar e_1}$, so $\bann(\NF_n)=\spn{B_n}$ has dimension one.
\end{remark}
\begin{proposition}
If $n\ge3$, none of the maps $x\circ y$, $x\circ y-y\circ x$, and $x\circ y+y\circ x$ is a biderivation of $\NF_n$.
\end{proposition}
\begin{proof}
The product is not symmetric: $e_1\circ e_2=e_3$ and $e_2\circ e_1=2e_3$. Its commutator is a nonzero skew-symmetric map. Both contradict Theorem~\ref{null} if they are biderivations. The anticommutator takes $(e_1,e_1)$ to $2e_2$, so a biderivation with this value must be $2B_2$. At $(e_1,e_2)$ the anticommutator is $3e_3$, whereas $2B_2(e_1,e_2)=6e_3$.
\end{proof}
\section{The standard filiform families}
For $n\ge5$, every complex filiform Zinbiel algebra is isomorphic to exactly one of the following algebras \cite[Theorem~3.5]{AKO}:
\begin{align*}
F_n^1:&\quad e_i\circ e_j=\binom{i+j-1}{j}e_{i+j},\quad 2\le i+j\le n-1;\\
F_n^2:&\quad\text{the products of }F_n^1,\quad e_n\circ e_1=e_{n-1};\\
F_n^3:&\quad\text{the products of }F_n^1,\quad e_n\circ e_n=e_{n-1}.
\end{align*}
These tables define Zinbiel algebras for $n\ge3$, but the cited completeness statement concerns $n\ge5$. We first compute their biderivations directly from the tables.
\subsection{A reduction using the annihilator}
Put $m=n-1$, $N=\NF_m$, and $f=e_n$. For $\alpha,\beta\in\K$, write $F_{\alpha,\beta}=N\oplus\K f$ with
\begin{equation}\label{fab}(x+sf)\circ(y+tf)=x\circ_N y+(\alpha sy_1+\beta st)e_m.\end{equation}
Here $x=\sum_{i=1}^m x_ie_i$, and similarly for $y$. The three families correspond to $(\alpha,\beta)=(0,0),(1,0),(0,1)$. The correction is annihilator-valued and vanishes whenever either input is a product; thus direct substitution in \eqref{zin} shows that \eqref{fab} is Zinbiel. For $m\ge2$, its square is $\spn{e_2,\ldots,e_m}$ and its higher powers are $\spn{e_k,\ldots,e_m}$ for $2\le k\le m$. It is filiform and generated by $e_1,f$.

Let $B_k^N$ denote extension of \eqref{Bk} on $N$ by zero on pairs involving $f$. For $w_{11},w_{1f},w_{f1},w_{ff}\in\Ann(F_{\alpha,\beta})$, define
\begin{equation}\label{Phi}\Phi(x+sf,y+tf)=x_1y_1w_{11}+x_1tw_{1f}+sy_1w_{f1}+stw_{ff}.\end{equation}
We also write $\Phi=\Phi(w_{11},w_{1f},w_{f1},w_{ff})$.
\begin{lemma}\label{annf}
For $m\ge2$,
\[\Ann(F_{\alpha,\beta})=\{ce_m+tf:\alpha t=\beta t=0\}.\]
If $v$ satisfies $e_1\circ v=v\circ e_1=0$ and $v\circ f+f\circ v=0$, then $v$ belongs to this annihilator. If $D$ is a derivation with $D(e_1)=0$, then $D(f)$ belongs to it.
\end{lemma}
\begin{proof}
Write $v=y+tf$. Since $e_1\circ y=\sum_{i=1}^{m-1}y_ie_{i+1}$, its vanishing forces $y=ce_m$. The equation $v\circ e_1=0$ gives $\alpha t=0$, and $v\circ f+f\circ v=0$ gives $2\beta t=0$. These conditions imply that $v$ annihilates the whole algebra on both sides. If $D(e_1)=0$, then $D$ vanishes on $N$, which $e_1$ generates. Applying $D$ to $e_1\circ f=0$, $f\circ e_1=\alpha e_m$, and $f\circ f=\beta e_m$ gives the three equations for $v=D(f)$.
\end{proof}
\begin{lemma}\label{extend}
Each map \eqref{Phi} is a biderivation. Also $B_k^N$ is a biderivation for $k\ge2$, and for every $k$ if $\alpha=\beta=0$.
\end{lemma}
\begin{proof}
The first assertion follows from Theorem~\ref{ann}, since the classes of $e_1,f$ form a basis of $F_{\alpha,\beta}/F_{\alpha,\beta}^2$. For $k\ge2$, \eqref{Bk} gives $B_k(e_m,z)=0$ and shows that every value has zero $e_1$-coordinate. In \eqref{bi1}, the additional product on the left is therefore killed by $B_k^N$, and the additional term on the right is also zero. The remaining equality is the identity on $N$. Symmetry gives \eqref{bi2}. For $\alpha=\beta=0$, extension by zero is immediate for every $k$.
\end{proof}
\begin{proposition}\label{reduce}
If $B\in\Bider(F_{\alpha,\beta})$ and $B(e_1,e_1)=0$, then
\[B=\Phi(0,u,v,w),\qquad u,v,w\in\Ann(F_{\alpha,\beta}).\]
\end{proposition}
\begin{proof}
Apply Lemma~\ref{annf} to $B(e_1,\cdot)$ and $B(\cdot,e_1)$ to obtain $u=B(e_1,f)$ and $v=B(f,e_1)$ in the annihilator. Subtract $\Phi(0,u,v,0)$. The difference vanishes at $(f,e_1)$; applying the same lemma to its derivation in the second argument with first argument $f$ yields $w=B(f,f)\in\Ann(F_{\alpha,\beta})$. The displayed map agrees with $B$ on all generator pairs. Proposition~\ref{generators} completes the proof.
\end{proof}
\subsection{The split family}
Here $F_n^1=N\oplus\K f$ and $\Ann(F_n^1)=\spn{e_m,f}$.
\begin{theorem}\label{f1}
For $n\ge3$, all biderivations of $F_n^1$ are uniquely given by
\[B=\sum_{k=1}^m r_kB_k^N+\Phi(\rho f,u_1e_m+u_2f,v_1e_m+v_2f,w_1e_m+w_2f).\]
The parameters are arbitrary. The map is symmetric exactly when $(u_1,u_2)=(v_1,v_2)$, and skew-symmetric exactly when $r_k=\rho=w_1=w_2=0$ and $(v_1,v_2)=-(u_1,u_2)$. Hence
\[(\dim\Bider(F_n^1),\dim\bsym(F_n^1),\dim\bskew(F_n^1))=(n+6,n+4,2).\]
A basis of the skew-symmetric part is $E_{1,n}^{e_m}-E_{n,1}^{e_m}$, $E_{1,n}^{f}-E_{n,1}^{f}$.
\end{theorem}
\begin{proof}
All displayed maps are biderivations by Lemma~\ref{extend}. For arbitrary $B$, write $B(e_1,e_1)=\sum_{k=1}^m r_ke_k+\rho f$ and subtract $\sum_{k=1}^m r_kB_k^N+\Phi(\rho f,0,0,0)$. Proposition~\ref{reduce} gives the remaining terms. The four generator pairs recover every parameter uniquely. Comparing their values with the transposes gives the symmetry conditions, and counting parameters gives the dimensions.
\end{proof}
\subsection{The family with \texorpdfstring{$f\circ e_1=e_m$}{f o e1 = em}}
For $F_n^2$, assume $n\ge4$, so $m\ge3$. Define symmetric maps
\begin{align}
H(x+sf,y+tf)&=B_1(x,y)+(m-1)\big((x_2t+sy_2)e_m+(x_1t+sy_1)f\big),\label{H}\\
R(x+sf,y+tf)&=(x_2y_1+x_1y_2)e_m+x_1y_1f.\label{R}
\end{align}
\begin{theorem}\label{f2}
Every biderivation of $F_n^2$, $n\ge4$, has a unique expression
\[B=r_1H+\sum_{k=2}^m r_kB_k^N+\rho R+aE_{1,n}^{e_m}+bE_{n,1}^{e_m}+cE_{n,n}^{e_m}.\]
It is symmetric exactly when $a=b$. The skew-symmetric part is spanned by $E_{1,n}^{e_m}-E_{n,1}^{e_m}$, and the dimension triple is $(n+3,n+2,1)$.
\end{theorem}
\begin{proof}
We first verify $H,R$. Both are symmetric, so it suffices to check \eqref{bi1}. Use $x+sf,y+tf,z+pf$ and subtract the identity for $B_1$ on $N$. For $H$, the residual $e_m$-coefficients are
\begin{align*}
\text{left: }&msy_1z_1+(m-1)x_1y_1p,\\
\text{right: }&(m-1)(x_1p+sz_1)y_1+sy_1z_1.
\end{align*}
They agree; the residual $f$-components are zero. Here $m\ge3$ makes the first two coordinates of $e_m$ zero. For $R$, both sides reduce to $x_1y_1z_1e_m$. The remaining maps are biderivations by Lemma~\ref{extend}, since $\Ann(F_n^2)=\K e_m$.

For necessity, write $B(e_1,e_1)=\sum_{k=1}^m r_ke_k+\rho f$ and subtract $r_1H+\sum_{k=2}^m r_kB_k^N+\rho R$. The difference vanishes at $(e_1,e_1)$, so Proposition~\ref{reduce} gives the last three terms. The generator values are
\begin{align*}
B(e_1,e_1)&=\sum_{k=1}^m r_ke_k+\rho f,\\
B(e_1,f)&=ae_m+(m-1)r_1f,\\
B(f,e_1)&=be_m+(m-1)r_1f,\qquad B(f,f)=ce_m.
\end{align*}
They prove uniqueness, the symmetry assertions, and the dimensions.
\end{proof}
\subsection{The family with \texorpdfstring{$f\circ f=e_m$}{f o f = em}}
\begin{theorem}\label{f3}
For $n\ge4$, every biderivation of $F_n^3$ is uniquely given by
\[B=\sum_{k=2}^m r_kB_k^N+aE_{1,n}^{e_m}+bE_{n,1}^{e_m}+cE_{n,n}^{e_m}.\]
It is symmetric exactly when $a=b$. Its skew-symmetric part is spanned by $E_{1,n}^{e_m}-E_{n,1}^{e_m}$, and its dimension triple is $(n+1,n,1)$.
\end{theorem}
\begin{proof}
Sufficiency follows from Lemma~\ref{extend}. Every derivation $D$ of $F_n^3$ has no $f$-component in $D(e_1)$. Indeed, write $D(e_1)=p+\gamma f$ and $D(f)=y+tf$, with $p,y\in N$. Differentiating $e_1\circ f=f\circ e_1=0$ gives
\[e_1\circ y=-\gamma e_m,\qquad y\circ e_1=-\gamma e_m.\]
The first forces $y=-\gamma e_{m-1}+y_me_m$, and the second gives $(2-m)\gamma=0$. Since $m\ge3$, $\gamma=0$.

The grading map $G(e_i)=ie_i$ for $1\le i\le m$, $G(f)=\frac m2f$ is a derivation, also respecting $f\circ f=e_m$. Write $B(e_1,e_1)=\sum_{k=1}^m r_ke_k+\gamma f$. Applying the previous observation to $B(\cdot,e_1)$ gives $\gamma=0$. Subtract $\sum_{k=2}^m r_kB_k^N$, obtaining $C(e_1,e_1)=r_1e_1$. The derivation $C(\cdot,e_1)-r_1G$ kills $e_1$, so Lemma~\ref{annf} implies that the $f$-component of $C(f,e_1)$ equals $mr_1/2$. But $C(f,\cdot)$ is a derivation whose value at $e_1$ has zero $f$-component. Thus $r_1=0$. Proposition~\ref{reduce} applies. The generator values give uniqueness and the symmetry claims.
\end{proof}

\section{Filiform algebras in dimension three}
Every three-dimensional filiform algebra has $\dim A^2=1$ and $A^3=0$, so $A^2\subseteq\Ann(A)$. Choose $A=V\oplus\K c$, where $\dim V=2$ and $A^2=\K c$. Its product is
\begin{equation}\label{form}(x+sc)\circ(y+tc)=\varphi(x,y)c,\end{equation}
where $\varphi$ is a nonzero bilinear form. Conversely, every such product is Zinbiel and filiform. Denote this algebra by $A_\varphi$.
\subsection{Classification by bilinear forms}
\begin{theorem}\label{class3}
Every three-dimensional complex filiform Zinbiel algebra is isomorphic to one of
\begin{align*}
F_3^1:&\quad e_1\circ e_1=e_2;\\
A_3(\mu):&\quad e_1\circ e_3=e_2,\quad e_3\circ e_1=\mu e_2,\quad\mu\in\C;\\
G_3:&\quad e_1\circ e_1=e_2,\quad e_1\circ e_3=e_2,\quad e_3\circ e_1=-e_2.
\end{align*}
The first and third are distinct from every member of the middle family, and
\[A_3(\mu)\cong A_3(\mu')\quad\Longleftrightarrow\quad\mu'=\mu\ \text{or}\ \mu\mu'=1.\]
Moreover, $F_3^2\cong A_3(0)$ and $F_3^3\cong A_3(1)$.
\end{theorem}
\begin{proof}
An isomorphism $A_\varphi\to A_\psi$ preserves $A^2=\K c$ and induces an invertible map $P$ on $V$. If $c\mapsto\kappa c$, preservation of multiplication is equivalent to $\psi(Px,Py)=\kappa\varphi(x,y)$, with $\kappa\ne0$. Conversely, this gives an isomorphism. The problem is therefore congruence of nonzero two-dimensional bilinear forms up to a nonzero scalar.

Write $\varphi=\sigma+\omega$ with $\sigma$ symmetric and $\omega$ alternating. If $\omega=0$, the rank of $\sigma$ is one or two, yielding $F_3^1$ or $A_3(1)$. If $\sigma=0$, the form is nondegenerate alternating, yielding $A_3(-1)$. If both parts are nonzero and $\rank\sigma=1$, choose $\sigma(x,y)=x_1y_1$ and scale the second vector so that $\omega(x,y)=x_1y_3-x_3y_1$. This gives $G_3$. If $\rank\sigma=2$, choose $\sigma(x,y)=x_1y_3+x_3y_1$. Then
\[\varphi(x,y)=(1+k)x_1y_3+(1-k)x_3y_1.\]
Scaling, and interchanging basis vectors if $k=-1$, yields $A_3(\mu)$.

For a form matrix $M$, congruence and scalar multiplication change $\det(sM+tM^T)$ by a nonzero factor. For $F_3^1$ it is identically zero; for $A_3(\mu)$ it is $-(s+\mu t)(\mu s+t)$; for $G_3$ it is $(s-t)^2$. The double root shared with $A_3(-1)$ does not identify them, since only the latter form is alternating. Comparing the projective roots for $A_3(\mu)$ gives equal or reciprocal parameters. Conversely, interchanging $e_1,e_3$ and rescaling $e_2$ gives the reciprocal isomorphism if $\mu\ne0$.

The maps from $A_3(0)$ to $F_3^2$ given by $e_1\mapsto e_3$, $e_3\mapsto e_1-e_3$, $e_2\mapsto e_2$, and from $A_3(1)$ to $F_3^3$ given by
\[e_1\mapsto e_1+\mathrm i e_3,\qquad e_3\mapsto\tfrac12(e_1-\mathrm i e_3),\qquad e_2\mapsto e_2\]
are invertible and preserve all products.
\end{proof}
\subsection{Biderivations and conformal endomorphisms}
For a nonzero form $\varphi$ on $V$, define
\[\rad\varphi=\{v:\varphi(v,V)=\varphi(V,v)=0\},\]
\[\mathfrak{co}(\varphi)=\{X\in\End(V):\varphi(Xx,y)+\varphi(x,Xy)=\lambda(X)\varphi(x,y)\}.
\]
Since $\varphi\ne0$, the scalar $\lambda(X)$ is uniquely determined and linear in $X$.
\begin{theorem}\label{conformal}
Suppose $\rad\varphi=0$. Then $\Ann(A_\varphi)=\K c$. Biderivations of $A_\varphi$ correspond bijectively to pairs $(T,\tau)$, where $T:V\times V\to V$ and $\tau:V\times V\to\K$ are bilinear, $\tau$ is arbitrary, and $T(\cdot,y),T(x,\cdot)\in\mathfrak{co}(\varphi)$ for all $x,y\in V$. The correspondence is
\begin{align*}
B(x,y)&=T(x,y)+\tau(x,y)c,\\
B(c,y)&=\lambda(T(\cdot,y))c,\qquad B(x,c)=\lambda(T(x,\cdot))c,\qquad B(c,c)=0.
\end{align*}
The map $B$ is symmetric or skew-symmetric exactly when both $T$ and $\tau$ have that property.
\end{theorem}
\begin{proof}
The radical hypothesis gives the annihilator assertion. Proposition~\ref{preserve} makes $B(c,y),B(x,c)$ multiples of $c$. For $x,y,z\in V$, \eqref{bi1} is
\[\varphi(x,y)B(c,z)=\big(\varphi(T(x,z),y)+\varphi(x,T(y,z))\big)c.\]
Choose $x,y$ with $\varphi(x,y)\ne0$ to identify the coefficient of $B(c,z)$ and obtain $T(\cdot,z)\in\mathfrak{co}(\varphi)$. Equation~\eqref{bi2} gives the other partial condition. Taking the third argument in \eqref{bi1} to be $c$ gives $B(c,c)=0$.

Conversely, the displayed definition satisfies the identities on triples from $V$ by the conformal conditions. If a product input is $c$, the relevant products vanish because $c$ and the corresponding values are in the annihilator. If the unused argument is $c$, both sides vanish by $B(c,c)=0$. Trilinearity completes the verification. Projection onto $V$ and $\K c$ gives uniqueness. Interchanging inputs and using linearity of $\lambda$ proves the symmetry assertion.
\end{proof}
\begin{lemma}\label{regular}
Let $W\subseteq\End(V)$ be a commutative associative subalgebra containing the identity. Suppose $u\in V$ makes $W\to V$, $X\mapsto Xu$, bijective. Then $V$ becomes a commutative associative algebra with unit $u$ and multiplication $x\cdot y=X_xy$, where $X_xu=x$. Bilinear maps with all partial operators in $W$ are exactly
\[T_w(x,y)=w\cdot x\cdot y,\qquad w\in V.\]
They are symmetric and form a space of dimension $\dim V$.
\end{lemma}
\begin{proof}
Since $X_xX_y\in W$ and these operators commute, $X_xy=X_xX_yu=X_yx$. Identifying $V$ with $W$ through evaluation at $u$ proves associativity. Every operator in $W$ is multiplication by a unique element. A bilinear map with partial operators in $W$ therefore satisfies $T(x,y)=g(y)\cdot x=h(x)\cdot y$. Set $x=u$ to obtain $g(y)=h(u)\cdot y$, so $T=T_{h(u)}$. The converse is immediate.
\end{proof}
\begin{lemma}\label{cos}
Relative to the ordered basis $(e_1,e_3)$ of $V$:
\begin{enumerate}
\item For $A_3(\mu)$, $\mu\ne-1$, $\mathfrak{co}(\varphi)$ is the diagonal algebra and $\lambda(\operatorname{diag}(a,d))=a+d$.
\item For $A_3(-1)$, $\mathfrak{co}(\varphi)=\End(V)$ and $\lambda(X)=\tr X$.
\item For $G_3$, $\mathfrak{co}(\varphi)=\spn{I,N}$, where $Ne_1=e_3$, $Ne_3=0$, and $\lambda(pI+qN)=2p$.
\item For $F_3^3$, $\mathfrak{co}(\varphi)=\spn{I,J}$, where $Je_1=e_3$, $Je_3=-e_1$, and $\lambda(pI+qJ)=2p$.
\item For $F_3^2$, it is the diagonal algebra in $u_1=e_3$, $u_2=e_1-e_3$, with $\lambda(X)=\tr X$.
\end{enumerate}
\end{lemma}
\begin{proof}
Write $Xe_1=ae_1+be_3$, $Xe_3=ge_1+de_3$. For $A_3(\mu)$, the diagonal entries of $X^TM+MX$ are $(1+\mu)b,(1+\mu)g$, and its off-diagonal entries are $a+d,\mu(a+d)$. Thus $b=g=0$ for $\mu\ne-1$; otherwise there is no restriction and the multiplier is $a+d$. For $G_3$, the four entries give $g=0,d=a$, with multiplier $2a$. For $F_3^3$, the equation is $X^T+X=\lambda I$, giving $d=a,g=-b$. For $F_3^2$, the form is $(x_1+x_3)y_1$. The independent linear forms $x\mapsto x_1+x_3$ and $x\mapsto x_1$ must be eigenvectors for the dual operator, equivalently the operator is diagonal in their dual basis $(u_1,u_2)$.
\end{proof}
\subsection{Explicit bases and the exceptional alternating case}
Write $E_{pq}=E_{p,q}^{e_2}$ for $p,q\in\{1,3\}$.
\begin{theorem}\label{three}
The following descriptions give all biderivations.
\begin{enumerate}
\item For $A_3(\mu)$, $\mu\ne-1$, a basis is $D_1,D_3,E_{11},E_{13},E_{31},E_{33}$, where
\begin{align*}
D_1(e_1,e_1)&=e_1,&D_1(e_1,e_2)=D_1(e_2,e_1)&=e_2,\\
D_3(e_3,e_3)&=e_3,&D_3(e_3,e_2)=D_3(e_2,e_3)&=e_2.
\end{align*}
The dimension triple is $(6,5,1)$ and $\bskew=\spn{E_{13}-E_{31}}$.
\item For $A_3(-1)$ and $p,q,r\in\{1,3\}$, define $M_{pqr}$ by
\[M_{pqr}(e_p,e_q)=e_r,\quad M_{pqr}(e_2,e_q)=\delta_{pr}e_2,\quad M_{pqr}(e_p,e_2)=\delta_{qr}e_2.\]
The eight maps $M_{pqr}$ and four maps $E_{pq}$ form a basis. The dimension triple is $(12,9,3)$, with skew-symmetric basis
\[M_{131}-M_{311},\qquad M_{133}-M_{313},\qquad E_{13}-E_{31}.\]
\item For $G_3$, a basis is $P_1,P_2,E_{11},E_{13},E_{31},E_{33}$, where
\begin{gather*}
P_1(e_1,e_1)=e_1,\quad P_1(e_1,e_3)=P_1(e_3,e_1)=e_3,\\
P_1(e_1,e_2)=P_1(e_2,e_1)=2e_2,\qquad P_2=E_{1,1}^{e_3}.
\end{gather*}
The dimension triple is $(6,5,1)$ and $\bskew=\spn{E_{13}-E_{31}}$.
\end{enumerate}
The case $F_3^1$ is covered by Theorem~\ref{f1} and has dimension triple $(9,7,2)$.
\end{theorem}
\begin{proof}
In (1)--(3) the two-sided radical is zero, so Theorem~\ref{conformal} applies. The arbitrary scalar form $\tau$ supplies the four $E_{pq}$.
In (1), the conformal algebra is diagonal. Lemma~\ref{regular} applies with $u=e_1+e_3$ and coordinatewise multiplication on $V\cong\K\times\K$. Hence $T(e_i,e_j)=\delta_{ij}w_ie_i$ for $i,j\in\{1,3\}$, giving two independent symmetric maps. The multipliers are one on the matching coordinate and zero on the other, so their extensions are $D_1,D_3$. Adding $\tau$ gives dimensions $6$, $2+3=5$, and $1$.

In (2), every bilinear $T:V\times V\to V$ is allowed. These maps form an eight-dimensional space with elementary basis $T_{pqr}(e_p,e_q)=e_r$. Partial operators have traces $\delta_{pr}$ and $\delta_{qr}$ at the indicated inputs, producing $M_{pqr}$. Symmetric $T$ have dimension six, skew-symmetric $T$ dimension two. Adding the three symmetric and one skew-symmetric scalar forms yields $(12,9,3)$ and the stated basis.

In (3), Lemma~\ref{regular} applies with $u=e_1$ and $e_3\cdot e_3=0$. Choosing $w=e_1,e_3$ yields the $V$-parts of $P_1,P_2$. Their multipliers give the displayed values involving $e_2$. These maps are symmetric, so the dimensions follow as in (1). Uniqueness in Theorem~\ref{conformal} proves independence and completeness in all cases.
\end{proof}
\begin{example}\label{exception}
On $A_3(-1)$ let $S=M_{131}-M_{311}$. Its nonzero values are
\begin{gather*}
S(e_1,e_3)=e_1,\quad S(e_3,e_1)=-e_1,\\
S(e_2,e_3)=e_2,\quad S(e_3,e_2)=-e_2.
\end{gather*}
It is a skew-symmetric biderivation by Theorem~\ref{three}, but its image is not contained in $\Ann(A_3(-1))=\K e_2$. Thus annihilator-valuedness is not automatic even for filiform Zinbiel algebras.
\end{example}
\subsection{The original bases of \texorpdfstring{$F_3^2$ and $F_3^3$}{F32 and F33}}
The generic formulas in Section~4 require care in dimension three. We record the complete results in the original bases.
\begin{proposition}\label{original3}
In each of $F_3^2,F_3^3$, the maps $E_{11},E_{13},E_{31},E_{33}$ together with the following two symmetric maps form a basis of $\Bider$. For $F_3^2$, the extra maps are $S_1,S_2$, where
\begin{gather*}
S_1(e_1,e_1)=e_1,\qquad S_1(e_1,e_3)=S_1(e_3,e_1)=S_1(e_3,e_3)=e_3,\\
S_1(e_1,e_2)=S_1(e_2,e_1)=2e_2,\qquad S_1(e_2,e_3)=S_1(e_3,e_2)=e_2,
\end{gather*}
and
\[S_2(x,y)=(x_1+x_3)(y_1+y_3)e_3+\big((x_1+x_3)y_2+x_2(y_1+y_3)\big)e_2.\]
For $F_3^3$, the extra maps are $P,Q$, where
\begin{gather*}
P(e_1,e_1)=e_1,\quad P(e_1,e_3)=P(e_3,e_1)=e_3,\quad P(e_3,e_3)=-e_1,\\
P(e_1,e_2)=P(e_2,e_1)=2e_2,\\[2pt]
Q(e_1,e_1)=e_3,\quad Q(e_1,e_3)=Q(e_3,e_1)=-e_1,\quad Q(e_3,e_3)=-e_3,\\
Q(e_3,e_2)=Q(e_2,e_3)=-2e_2.
\end{gather*}
Both have dimension triple $(6,5,1)$ and skew-symmetric basis $E_{13}-E_{31}$.
\end{proposition}
\begin{proof}
For $F_3^2$, use the conformal algebra diagonal in $u_1=e_3,u_2=e_1-e_3$. Lemma~\ref{regular} identifies $V$ with the algebra of orthogonal idempotents $u_1,u_2$ and unit $e_1=u_1+u_2$. The choices $w=e_1,u_1$ give the $V$-parts of $S_1,S_2$. Their traces give the values involving $e_2$ via Theorem~\ref{conformal}. For $F_3^3$, use the algebra with unit $e_1$ and $e_3\cdot e_3=-e_1$. The choices $w=e_1,e_3$ give $P,Q$, with their values involving $e_2$ supplied by the multipliers. In both cases there are two independent symmetric $T$-parts and four arbitrary scalar forms, proving completeness and the dimensions.
\end{proof}
\begin{remark}
The dimension formula for $F_n^2$ remains valid at $n=3$, but $H,R$ in \eqref{H}--\eqref{R} do not specialize to biderivations. For $H$, \eqref{bi1} at $(e_3,e_1,e_3)$ has left side $e_2$ and right side zero. For $R$, the failure occurs at $(e_3,e_1,e_1)$. For $F_n^3$, even the dimension formula fails at $n=3$: the four generic directions must be supplemented by $P,Q$. The argument excluding an $f$-component uses $(2-m)\gamma=0$, which gives no restriction when $m=2$.
\end{remark}

\section{Filiform algebras in dimension four}
Two additional algebras occur. For $c\in\{0,1\}$ define $Z_4^c$ by
\begin{equation}\label{Z4}e_1\circ e_4=e_2,\quad e_4\circ e_1=-e_2,\quad e_1\circ e_2=e_3,\quad e_4\circ e_4=ce_3.\end{equation}
Their powers are $A^2=\spn{e_2,e_3}$, $A^3=\K e_3$, $A^4=0$.
\subsection{Completeness of the list}
\begin{theorem}\label{class4}
Every four-dimensional complex filiform Zinbiel algebra is isomorphic to exactly one of $F_4^1,F_4^2,F_4^3,Z_4^0,Z_4^1$.
\end{theorem}
\begin{proof}
Choose $A^3=\K e_3$, $A^2=\spn{e_2,e_3}$ and a complement $V$ of $A^2$. Since $A^4=0$, $A^3\subseteq\Ann(A)$ and $A^2\circ A^2=0$. There are linear forms $\lambda,\rho$ and bilinear forms $\varphi,\chi$ on $V$ such that
\[x\circ e_2=\lambda(x)e_3,\quad e_2\circ x=\rho(x)e_3,\quad x\circ y=\varphi(x,y)e_2+\chi(x,y)e_3.\]
Here $\lambda\ne0$ because $A^3=A\circ A^2$, and $\varphi\ne0$ because $\dim A^2=2$. The only potentially nontrivial Zinbiel identities are on $x,y,z\in V$:
\begin{equation}\label{bilinear4}\varphi(x,y)\rho(z)=\lambda(x)\big(\varphi(y,z)+\varphi(z,y)\big).\end{equation}
Suppose first $\rho\ne0$. Choose $x_0$ with $\lambda(x_0)=1$ and put $\nu(y)=\varphi(x_0,y)$. Equation~\eqref{bilinear4} gives $\varphi=\lambda\otimes\nu$ and $\varphi+\varphi^\tau=\nu\otimes\rho$. Symmetry implies $\nu=\kappa\rho$ with $\kappa\ne0$. Substitution yields
\[\lambda\otimes\rho+\rho\otimes\lambda=\rho\otimes\rho,\]
so $\rho=2\lambda$ and $\varphi$ has rank one. Choose $e_1\in V$ with $\lambda(e_1)=1$, $e_4\in\ker\lambda$, and redefine $e_2=e_1\circ e_1$, $e_3=e_1\circ e_2$. Then
\[e_1\circ e_1=e_2,\quad e_1\circ e_2=e_3,\quad e_2\circ e_1=2e_3,\]
with remaining possible products $e_1\circ e_4=ae_3$, $e_4\circ e_1=be_3$, $e_4\circ e_4=ge_3$. Replacing $e_4$ by $e_4-ae_2$ gives $a=0$. Next replace $e_1$ by $e_1+te_4$, define new $e_2,e_3$ from its successive products, and replace $e_4$ by $e_4-tge_2$. The remaining parameters become $(b-tg,g)$. If $g\ne0$, choose $t=b/g$ and rescale to get $F_4^3$. If $g=0,b\ne0$, rescale $e_4$ to get $F_4^2$; if both are zero, get $F_4^1$.

Now suppose $\rho=0$. Equation~\eqref{bilinear4} gives $\varphi+\varphi^\tau=0$, so $\varphi$ is nonzero alternating. Choose $e_1$ with $\lambda(e_1)=1$, nonzero $e_4\in\ker\lambda$, and redefine $e_2=e_1\circ e_4$, $e_3=e_1\circ e_2$. Then
\[e_4\circ e_1=-e_2+\delta e_3,\quad e_1\circ e_1=ae_3,\quad e_4\circ e_4=ce_3,\]
with $e_2\circ A=0$. Replacing $e_4$ by $e_4-\delta e_2$ and redefining $e_2$ removes $\delta$. Replacing $e_1$ by $e_1-ae_2$ removes $a$. Scaling $e_1,e_4$ changes $c$ by an arbitrary nonzero factor, giving $c=0$ or $1$. These tables satisfy \eqref{bilinear4}, hence are Zinbiel.

Finally, $A^2\circ A=0$ separates the two $Z$-algebras from the three $F$-algebras. The annihilator of $F_4^1$ has dimension two, those of $F_4^2,F_4^3$ dimension one. Put $U=\{x:x\circ A^2=0\}$. In each remaining pair, the invariant condition $x\circ x=0$ for every $x\in U$ holds for $Z_4^0$ or $F_4^2$, and fails for $Z_4^1$ or $F_4^3$. This proves pairwise nonisomorphism.
\end{proof}
\subsection{Biderivations of the additional algebras}
\begin{lemma}\label{Zder}
A linear map on $Z_4^c$ with $D(e_1)=\sum_{i=1}^4 a_ie_i$ and $D(e_4)=\sum_{i=1}^4 b_ie_i$ is a derivation if and only if
\[a_2=0,\quad b_1=0,\quad b_2=-2ca_4,\quad c(b_4-2a_1)=0,\]
and
\[D(e_2)=(a_1+b_4)e_2+(ca_4+b_2)e_3,\qquad D(e_3)=(2a_1+b_4)e_3.\]
\end{lemma}
\begin{proof}
The values at $e_2=e_1\circ e_4$ and $e_3=e_1\circ e_2$ are forced. Apply the derivation identity to $e_1\circ e_1=0$ to get $a_2=0$, to $e_4\circ e_2=0$ to get $b_1=0$, to $e_4\circ e_1=-e_2$ to get $b_2=-2ca_4$, and to $e_4\circ e_4=ce_3$ for the last restriction. Conversely, these conditions verify all four nonzero products and the zero products just used. Every other pair contains $e_2$ on the left or $e_3$ as an input. Its identity holds because the stated images of $e_2,e_3$ lie in $\spn{e_2,e_3}$ and $\K e_3$, respectively, and $b_1=0$.
\end{proof}
\begin{theorem}\label{Zbider}
Write $E_{pq}=E_{p,q}^{e_3}$ for $p,q\in\{1,4\}$.
\begin{enumerate}
\item $\Bider(Z_4^1)=\bann(Z_4^1)$ has basis $E_{11},E_{14},E_{41},E_{44}$ and dimension triple $(4,3,1)$.
\item $\Bider(Z_4^0)$ has basis $E_{11},E_{14},E_{41},E_{44},X,Y$, where $X=E_{1,1}^{e_4}$ and
\[Y(e_1,e_1)=e_1,\quad Y(e_1,e_3)=Y(e_3,e_1)=e_3,\quad Y(e_1,e_4)=Y(e_4,e_1)=-e_4.\]
Its dimension triple is $(6,5,1)$.
\end{enumerate}
In both cases, the skew-symmetric part is $\spn{E_{14}-E_{41}}$.
\end{theorem}
\begin{proof}
We have $\Ann(Z_4^c)=\K e_3$ and $Z_4^c/(Z_4^c)^2=\spn{\bar e_1,\bar e_4}$, so Theorem~\ref{ann} gives the four $E_{pq}$. The maps $X,Y$ are symmetric. For $c=0$, Lemma~\ref{Zder} shows their partial maps are derivations: the nonzero ones for $X$ are $e_1\mapsto e_4$; for $Y$ they are $\operatorname{diag}(1,0,1,-1)$, $e_1\mapsto e_3$, and $e_1\mapsto-e_4$.

For necessity, set $u=B(e_1,e_1)$, $v=B(e_1,e_4)$, $w=B(e_4,e_1)$, $z=B(e_4,e_4)$. Apply Lemma~\ref{Zder} to the four derivations fixing either generator in either argument. The restrictions are
\begin{gather*}
u_2=v_1=v_2=w_1=w_2=z_1=0,\qquad cu_4=0,\\
z_2=-2cv_4=-2cw_4,\qquad c(v_4-2u_1)=c(w_4-2u_1)=cz_4=0.
\end{gather*}
The same formulas give
\begin{align*}
B(e_1,e_2)&=(u_1+v_4)e_2,& B(e_2,e_1)&=(u_1+w_4)e_2,\\
B(e_4,e_2)&=z_4e_2+(cw_4+z_2)e_3,&B(e_2,e_4)&=z_4e_2+(cv_4+z_2)e_3.
\end{align*}
Apply Lemma~\ref{Zder} to $B(\cdot,e_2)$ and $B(e_2,\cdot)$. The zero second coordinate of their values at $e_1$ gives $v_4=w_4=-u_1$, and their values at $e_4$ give $z_4=0$. Hence $-3cu_1=0$, so $cu_1=0$ and $z_2=0$. We obtain
\begin{align*}
u&=u_1e_1+u_3e_3+u_4e_4,& v&=-u_1e_4+v_3e_3,\\
w&=-u_1e_4+w_3e_3,& z&=z_3e_3,\qquad cu_1=cu_4=0.
\end{align*}
For $c=1$, only the four annihilator-valued directions remain. For $c=0$, the extra parameters are $u_1Y+u_4X$. These maps agree with $B$ on all generator pairs, so Proposition~\ref{generators} gives equality everywhere. The same values prove independence. Finally, $X,Y$ are symmetric, while the $E$-maps have symmetric dimension three and skew-symmetric dimension one.
\end{proof}

\section{Consequences and exact checks}
The results describe the symmetric and skew-symmetric parts for complex filiform Zinbiel algebras in every dimension.
\begin{corollary}\label{unique}
Let $A$ be a finite-dimensional complex filiform Zinbiel algebra. Then
\[\bskew(A)\subseteq\bann(A)\quad\Longleftrightarrow\quad A\not\cong A_3(-1).\]
For null-filiform algebras, $\bskew(A)=0$.
\end{corollary}
\begin{proof}
For $n\ge5$, use \cite[Theorem~3.5]{AKO} and Theorems~\ref{f1}--\ref{f3}. In dimensions three and four, completeness follows from Theorems~\ref{class3}, \ref{class4}; the biderivation descriptions show that all skew-symmetric directions are annihilator-valued except the two in Theorem~\ref{three}(2). Example~\ref{exception} proves failure there. In dimension two the filiform condition forces zero multiplication, so $\Ann(A)=A$. In dimension one the only Zinbiel product is zero: if $e\circ e=ae$, \eqref{zin} gives $a^2=2a^2$, hence $a=0$. The null-filiform assertion is Theorem~\ref{null}.
\end{proof}
\begin{remark}
Theorem~\ref{ann} classifies annihilator-valued biderivations of any algebra. It does not assert that every skew-symmetric biderivation is annihilator-valued. Corollary~\ref{unique} is a separate consequence specific to complex filiform Zinbiel algebras.
\end{remark}
\begin{table}[htbp]\centering
\begin{tabular}{@{}lrrrr@{}}\toprule
Algebra & $\dim\Ann A$ & $\dim\Bider A$ & $\dim\bsym A$ & $\dim\bskew A$\\\midrule
Zero algebra, dimension 1&1&1&1&0\\
Zero algebra, dimension 2&2&8&6&2\\
$F_3^1$&2&9&7&2\\
$A_3(\mu)$, $\mu\ne-1$&1&6&5&1\\
$A_3(-1)$&1&12&9&3\\
$G_3$&1&6&5&1\\
$F_4^1$&2&10&8&2\\
$F_4^2$&1&7&6&1\\
$F_4^3$&1&5&4&1\\
$Z_4^0$&1&6&5&1\\
$Z_4^1$&1&4&3&1\\\bottomrule
\end{tabular}
\caption{Filiform algebras in dimensions at most four, with the one-dimensional convention of Section~2. The parameter $\mu$ is considered modulo the reciprocal equivalence of Theorem~\ref{class3}.}
\end{table}
\subsection{The coefficient system}
For an explicit finite-dimensional check without symmetry assumptions, write
\[e_i\circ e_j=\sum_{s=1}^n c_{ij}^se_s,\qquad B(e_i,e_j)=\sum_{k=1}^n b_{ij}^ke_k.\]
Comparing the $e_t$-coordinates of \eqref{bi1}--\eqref{bi2} gives, for every $1\le i,j,\ell,t\le n$,
\begin{align}
\sum_{s=1}^n c_{ij}^sb_{s\ell}^t&=\sum_{k=1}^n b_{i\ell}^kc_{kj}^t+\sum_{k=1}^n b_{j\ell}^kc_{ik}^t,\label{system1}\\
\sum_{s=1}^n c_{j\ell}^sb_{is}^t&=\sum_{k=1}^n b_{ij}^kc_{k\ell}^t+\sum_{k=1}^n b_{i\ell}^kc_{jk}^t.\label{system2}
\end{align}
These conditions are necessary and sufficient because the identities are trilinear. Triples for which a product is zero must still be included, since their right-hand sides may impose constraints.

The supplementary script \texttt{verify\_biderivations.py} constructs this system and performs Gaussian elimination over $\mathbb Q$. If $E$ is the coefficient matrix, then $\dim\Bider(A)=n^3-\rank E$. Appending $b_{ij}^k-b_{ji}^k=0$ or $b_{ij}^k+b_{ji}^k=0$ gives the symmetric or skew-symmetric dimension. For rational tables, ranks are unchanged on extension from $\mathbb Q$ to $\C$.

The script tests $\NF_n$ for $1\le n\le6$, all three standard families for $3\le n\le6$, $A_3(\mu)$ at $\mu=-1,0,1,2$, and $G_3,Z_4^0,Z_4^1$. It checks the Zinbiel identities, the dimension triples, the explicit basis maps in the tested dimensions, and the skew-symmetric map in Example~\ref{exception}. The computed dimensions agree with the applicable theorems. These finite checks supplement the proofs; the arbitrary-dimensional and arbitrary-parameter assertions rest on the mathematical arguments above.

\section{Further questions}
The generator reduction in Proposition~\ref{generators} applies beyond the families studied here. For a larger class of nilpotent Zinbiel algebras, it is natural to seek structural restrictions on allowable values at generator pairs, rather than solve the full coefficient system for each multiplication table.

A second question is to characterize Zinbiel algebras satisfying $\bskew(A)\subseteq\bann(A)$. Theorem~\ref{ann} identifies the right-hand subspace explicitly, but Example~\ref{exception} shows that the Zinbiel identity alone does not force the inclusion. The construction in Theorem~\ref{conformal} suggests that conformal endomorphism algebras can detect additional skew-symmetric directions.

Finally, these computations can be investigated in positive characteristic. Our proofs use characteristic zero in the polynomial model, factorial formulas, and scalar divisions. A positive-characteristic classification requires separate arguments even when the multiplication tables remain meaningful.

\section*{Declarations}
\paragraph{Funding.} The authors received no specific grant from any funding agency in the public, commercial, or not-for-profit sectors for this research.
\paragraph{Author contributions.} Arislanbay Abdireymov and Bakhtiyor Yusupov contributed equally to the conceptualization, mathematical investigation, development and verification of the proofs, and preparation and revision of the manuscript.
\paragraph{Competing interests.} The authors declare no competing interests.
\paragraph{Data and code availability.} All mathematical results and multiplication tables needed to understand the proofs are included in this article. The supplementary material contains the exact-arithmetic verification script \texttt{verify\_biderivations.py} and its output. No external empirical datasets were used.
\paragraph{Ethics approval and consent.} This theoretical mathematical study involved no human participants, human data, or animals. Ethics approval and consent to participate are not applicable. Consent to publish personal participant data is not applicable.
\paragraph{Use of AI-assisted tools.} OpenAI ChatGPT was used in preparing this version to assist with language editing, LaTeX formatting, bibliographic organization, and preparation of supplementary verification code. The authors bear responsibility for the final content.
\paragraph{Corresponding author.} Bakhtiyor Yusupov; email: \href{mailto:baxtiyor_yusupov_93@mail.ru}{\texttt{baxtiyor\_yusupov\_93@mail.ru}}.

\end{document}